\documentclass[a4paper,11pt]{amsart}
\usepackage{caption}
\usepackage{subcaption}
\usepackage[english]{babel,varioref}
\usepackage{epsfig}
\usepackage{graphicx}
\usepackage{a4wide}
\usepackage{url}
\usepackage{srcltx}
\usepackage{color}
\usepackage{hyperref}
\usepackage[sort,nocompress]{cite}
\newtheorem{df}{Definition}[section]

\newtheorem{thm}[df]{Theorem}
\newtheorem*{ass}{Assumption}
\newtheorem{lem}[df]{Lemma}

\newtheorem{remark}[df]{Remark}

\begin{document}

\title[Harmonically Enriched Multiscale Coarse Space]{Analysis of a
  New Harmonically Enriched Multiscale Coarse Space for Domain
  Decomposition Methods}

\author{Martin J. Gander}\address{Section of Mathematics, University
  of Geneva, 1211 Geneva 4, Switzerland}\email{Martin.Gander@unige.ch} 

\author{Atle Loneland}\address{Department of Informatics, University
  of Bergen, 5020 Bergen, Norway} \email{Atle.Loneland@ii.uib.no}
\author{Talal Rahman}\address{Department of Computing, Mathematics and
  Physics, Bergen University College, 5020 Bergen, Norway}
\email{Talal.Rahman@hib.no} 


\begin{abstract}
  We propose a new, harmonically enriched multiscale coarse space
  (HEM) for domain decomposition methods. For a coercive high contrast
  model problem, we show how to enrich the coarse space so that the
  method is robust against any variations and discontinuities in the
  problem parameters both inside subdomains and across and along
  subdomain boundaries. We prove our results for an enrichment
  strategy based on solving simple, lower dimensional eigenvalue
  problems on the interfaces between subdomains, and we call the
  resulting coarse space the spectral harmonically enriched multiscale
  coarse space (SHEM). We then also give a variant that performs
  equally well in practice, and does not require the solve of
  eigenvalue problems, which we call non-spectral harmonically
  enriched multiscale coarse space (NSHEM). Our enrichment process
  naturally reaches the optimal coarse space represented by the full
  discrete harmonic space, which enables us to turn the method into a
  direct solver (OHEM). We also extensively test our new coarse spaces
  numerically, and the results confirm our analysis
\end{abstract}

\maketitle

\section{Introduction}

It is well known that domain decomposition methods which are based on
local communication between subdomains need the addition of a coarse
problem in order to be scalable, see for example the books
\cite{smith1996domain,Toselli:2005:DDM} and references therein.  The
coarse space components of the coarse problem can however do much more
than just make a method scalable: work and difficulties of the
underlying domain decomposition method can be transferred into
the coarse space, and an optimal coarse space can even make the
underlying domain decomposition method into a direct solver
\cite{gander2013new,Santugini2013Disc}, like optimal transmission conditions based on the Dirichlet to Neumann 
operator can make a Schwarz method into a direct solver, see \cite{Nataf:1995:FCD,Gander:2006:OSM} and
references therein for optimal Schwarz methods, and \cite{gander2011optimal} for transmission conditions which include an optimal coarse space component that leads to convergence in two iterations, independently of the number of subdomain and underlying equation. Such optimal transmission 
conditions and coarse space components are however very expensive to use, and in 
practice one approximates them. In the case of transmission conditions, this led to 
the new class of optimized Schwarz methods, which have the same computational
cost per iteration as classical Schwarz methods, but converge much faster; for an overview and references, see \cite{Gander:2006:OSM}. In the case of coarse spaces, one can select the most important 
subspace of the optimal coarse space to increase the performance of the method, 
an example of this has been shown in \cite{gander2013new}.


We present here a simple procedure that allows us to systematically
enrich a given coarse space, up to a maximum degree where it becomes a
direct solver.  Since our interest is in high contrast problems, we
start with a classical multiscale finite element coarse space. The
multiscale finite element method was developed to cope with problems that
have many spatial scales \cite{Hou:1997:MsFEM}. The idea is to replace
the classical finite element shape functions by harmonic functions,
i.e. functions that solve the underlying equation locally. An
important problem in multiscale finite element methods is what
boundary conditions one should impose on these harmonic shape
functions. The approach we propose here for our coarse space
enrichment can also be used to enrich a multiscale finite element
space, and if combined with non-overlapping subdomain solves, will in the limit reach the fine scale finite element
solution, which is an important property of our enrichment process.

Domain decomposition methods for problems where the discontinuities in the coefficient are resolved by the coarse mesh and the subdomains have been analyzed thoroughly, (cf. \cite{Dryja:1996:MultSchw,Dryja:1994:Schw,Mandel:1996:BDD,Sarkis:1997:NCS,Xu:1998:SomeNonOver,BJORSTAD:1996:AVG} and references therein). In the case where the discontinuous coefficients are not aligned with the coarse mesh or the underlying subdomains, efforts have been made to develop coarse spaces that would ultimately lead to robust methods with respect to the contrast in the
problems. We mention especially the use of coarse spaces consisting of multiscale finite element basis functions introduced in \cite{Hou:1997:MsFEM}. This approach was first studied in \cite{Aarnes:2002:DDmult} and later analyzed in \cite{Graham:2007:DDmult}, where the precise dependence of two level
Schwarz methods on the high contrast in a problem was given for the specific case of isolated inclusions. Also,
the importance of harmonic shape functions was shown for coarse grid
corrections. A finite volume multiscale coarse space was proposed in \cite{Nordbotten:2008:RelDDmult} and in \cite{Pechstein:2008:FETIMult,Pechstein:2011:FETIMult} FETI methods were analyzed for multiscale problems.


A first idea to enrich the coarse space by eigenfunctions for tackling problems with high contrast can be found in
\cite{Galvis:2010:DDmult,Galvis:2010:DDmult:RE}, where selected subdomain eigenfunctions
are combined with different types of partition of unity functions and
the importance of the initial coarse space based on multiscale and
energy minimizing basis functions is discussed, see also
\cite{Efendiev:2012:DDprecSPD}. A different way to construct a coarse
space using eigenfunctions of the Dirichlet to Neumann map of each
subdomain was proposed in \cite{Dolean:2012:DTN}.  This approach was
later improved by solving a generalized eigenvalue problem in the
overlaps \cite{Spillane:2014:GENEO}. A good overview of the most
recent approaches can be found in \cite{Scheichl:2013:ROBCOA}.

We start with a numerical experiment to motivate our new harmonic
enrichment process. The problem configuration is shown in Figure
\ref{fig:complicatedex}.
\begin{figure}
   \centering
   \includegraphics[width=\textwidth]{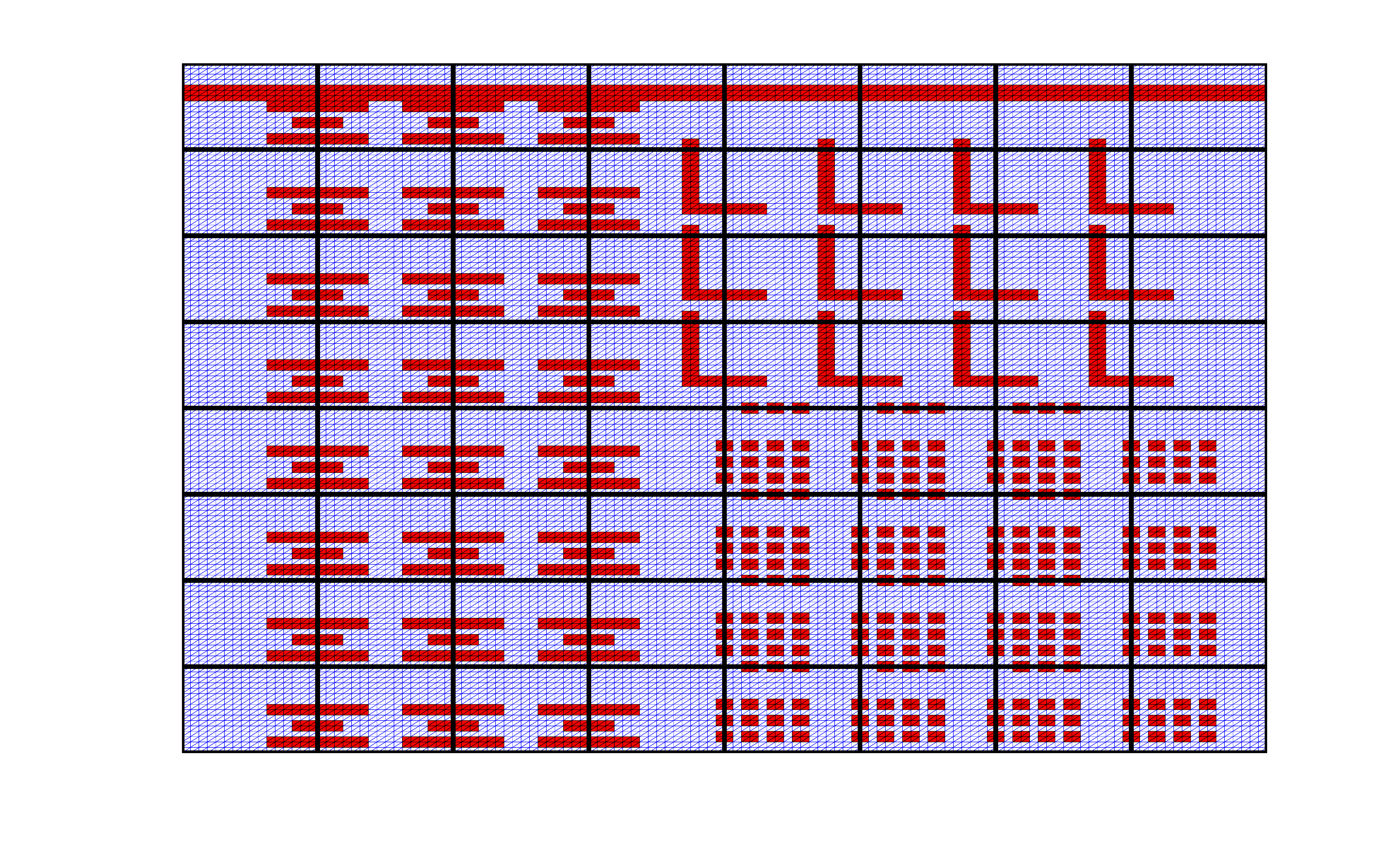}
  \caption{Distribution of $\alpha$ for a geometry with
    $h=\frac{1}{128}$, $H=16h$. The regions marked with red are where
    $\alpha$ has a large value.}
  \label{fig:complicatedex}
\end{figure}
We show in Table \ref{IntroTab1} the iteration numbers and condition
number estimates for a two level additive Schwarz method
when increasing the overlap, where the coarse space is the multiscale coarse space introduced in \cite{Graham:2007:DDmult}. 
\begin{table}\footnotesize
\centering
\small
\tabcolsep0.2em
\begin{tabular}{|c|c|c|c|c|c|c|c|}\hline
  $\delta=2h$&$\delta=4h$&$\delta=6h$&$\delta=8h$&$\delta=10h$&$\delta=12h$&$\delta=14h$&$\delta=16h$\\\hline
610(3.6e6)&339(1.1e6)&330(7.1e5)&241(3.5e5)&224(2.9e5)&222(2.5e5)&33(3.4e1)&30(2.8e1)\\\hline
\end{tabular}
\caption{Iteration and condition number estimate for the distribution
  in Figure~\ref{fig:complicatedex} for a classical two level additive
  Schwarz method when increasing the overlap $\delta$, with
  $h=\frac{1}{128}$, $H=16h$ and contrast $\alpha=10^6$.}
\label{IntroTab1}
\end{table}
We see that initially increasing the overlap improves the method, but
increasing further does not help much, until we reach the overlap
$\delta=14h$, where a substantial improvement happens. Why is this so?
Looking at Figure \ref{fig:complicatedex}, we see that there are many
highly conductive channels across the interfaces, which means that an
error component will travel with very little damping across these
channels. Following the original maximum principle argument of
Schwarz, the error will thus not diminish substantially, until the
overlap includes the entire channel, at which point rapid convergence
will set in. This is a typical case where the underlying domain
decomposition method has difficulties, and these difficulties can be
transferred into the coarse space. The present multiscale coarse space is however not good enough to handle
these difficulties.

The purpose of our manuscript is to show how one can systematically
transfer such a difficulty into the coarse space by enriching it with
well chosen harmonic functions to obtain a harmonically enriched
multiscale coarse space (HEM), and to prove that with
HEM the method is then robust with respect to the coefficient
jumps. In order to prove our results, we need to base the enrichment
on an eigenvalue problem on the interfaces between the subdomains,
which leads to the spectral harmonically enriched multiscale coarse
space (SHEM).  For SHEM, we have to solve lower
dimensional eigenvalue problems in the construction phase. Because
they are lower dimensional however, our construction process is much
cheaper than the construction based on volume subdomain eigenvalue
problems, Dirichlet to Neumann eigenvalue problems or generalized
eigenvalue problems in the overlap. We then show numerically that one
can obtain an equally good harmonic enrichment without eigenvalue
problems, which we call the non-spectral harmonically enriched
multiscale coarse space (NSHEM).


\section{Preliminaries}

\subsection{The Model Problem}
We consider as our model problem the elliptic boundary value problem
\begin{equation}\label{eq:modelproblem}
  \begin{array}{rcll}
    -\nabla\cdot(\alpha(x)\nabla u)&=&f \qquad &\mbox{in $\Omega$},\\
     u&=&0 &\mbox{on $\partial\Omega$},
  \end{array}
\end{equation}
where $\Omega$ is a bounded convex domain in $\mathbb{R}^2$, $f\in
L^2(\Omega)$ and $\alpha\in L^\infty(\Omega)$ with the property that
$\alpha\geq \alpha_0$ for some positive constant $\alpha_0$.

The weak formulation of (\ref{eq:modelproblem}) is: find $u \in
H^1_0(\Omega)$ such that
\begin{equation}\label{eq:weakformulation}
  a(u,v)=\int_\Omega f v \:dx \quad \forall  v \in H^1_0(\Omega),
\end{equation}
where 
$$
  a(u,v)= \int_{\Omega} \alpha(x)\nabla u\cdot\nabla v \: dx.
$$
We discretize (\ref{eq:weakformulation}) with standard $P_1$ finite
elements on a mesh $\mathcal{T}_h$ of $\Omega$ such that
$$
  \overline{\Omega}=\bigcup_{K\in\mathcal{T}_h}K,
$$
where the finite element space is defined as
$$
  V_h^0(\Omega):=\{v\in C_0(\Omega) : v|_K\in P_1(K)\}.
$$
The discrete problem corresponding to (\ref{eq:weakformulation}) is:
find $u_h\in V_h^0(\Omega)$ such that
\begin{equation} \label{eq:discreteformulation}
  a(u_h,v)=\int_\Omega f v \:dx \quad \forall  v \in V_h^0(\Omega).
\end{equation}
Without loss of generality we will assume that the coefficient $\alpha$ is piecewise constant over each element $K$, i.e., 
$$\alpha(x) = \alpha_K\mbox{ for all } x\in K.$$

We also introduce the following notation: for positive constants $c$
and $C$ independent of the mesh parameters $h$, $H$, the overlap
$\delta$ (which we will define below) and the coefficient $\alpha$,
we define $u\simeq v$, $x\succeq y$ and $w\preceq z$ as
$$
   cu\leq v\leq Cu,\qquad x\geq cy\qquad
    \text{and}\qquad w\leq Cz.
$$
Here, $u,v,x,y,w$ and $z$ denote norms of some functions.

\subsection{Subdomains}

Let $\Omega$ be partitioned into non-overlapping open, connected
Lipschitz polytopes $\{\Omega_i : i=1,\ldots,N\}$ such that
$\overline{\Omega}=\bigcup_{i=1}^N \overline{\Omega}_i$, where each
$\overline\Omega_i$ is assumed to consist of elements from
$\mathcal{T}_h(\Omega)$. We assume that this partitioning is
shape-regular. By extending each subdomain $\Omega_i$ with a distance
$\delta$ in each direction, we create a further decomposition of
$\Omega$ into overlapping subdomains $\{\Omega_i'\}_{i=1}^N$. We
consider here only the specific case of small overlap $\delta=2h$, but a slight modification of the proof will also cover the case of minimal overlap $h$. We
denote the layer of elements in $\Omega_i$ touching the boundary
$\partial\Omega_i$ by $\Omega^h_i$ and assume that the triangles 
corresponding to this layer are shape regular and define the minimum diameter of these triangles as
$h_i:=\min\limits_{K\in\mathcal{T}_h(\Omega_i^h)}h_K$, where $h_K$ is the diameter of the triangle $K$. We define the interface between two subdomains
to be the open edge shared by the subdomains, i.e., $\overline\Gamma_{ij}:=\overline\Omega_i\cap\overline\Omega_j$. The sets of vertices of elements in $\mathcal{T}_h(\Omega)$ (nodal points) belonging to
$\Omega$, $\Omega_i$, $\partial\Omega$, $\partial\Omega_i$ and $\Gamma_{ij}$ are denoted by
$\Omega_{h}$, $\Omega_{ih}$, $\partial\Omega_{h}$, $\partial\Omega_{ih}$ and $\Gamma_{ijh}$. With each interface we define the space of finite element functions restricted to $\Gamma_{ij}$ and zero on $\partial\Gamma_{ij}$ as $V_h^0(\Gamma_{ij})$.

The weighted $L^2$-inner product and the corresponding weighted $L^2$-norm is defined as 
$$
\left(u,v\right)_{L^2_\alpha(G)}:=\left(\alpha(x)u,v\right)_{L^2(G)}\quad\mbox{ and }\quad
\|u\|^2_{L^2_\alpha(G)}:=\left(u,u\right)_{L^2_\alpha(G)},
$$
where $G$ is some domain contained in $\Omega$.

We define the restriction of the bilinear form $a(\cdot,\cdot)$ to an
interface, $\Gamma_{ij}\subset\Gamma$, shared by two subdomains as
\begin{equation*}
a_{\Gamma_{ij}}(u,v):=\left(\alpha_{|\Gamma_{ij}}(x)D_{x^t}u,D_{x^t}v\right)_{L^2(\Gamma_{ij})},
\end{equation*}
where $\alpha_{|\Gamma_{ij}}(x):=\lim\limits_{y\in\Omega_i\rightarrow x}\alpha(y)$ and $D_{x^t}$ denotes the tangent derivative with respect to $\Gamma_{ij}$.  In order to obtain continuous basis functions across subdomain interfaces, we define a second bilinear form on each interface $\Gamma_{ij}$,
\begin{equation*}
\bar{a}_{\Gamma_{ij}}(u,v) :=\left(\overline\alpha_{ij}(x)D_{x^t}u,D_{x^t}v\right)_{L^2(\Gamma_{ij})},
\end{equation*}
where $\overline\alpha_{ij}$ is taken as the maximum of $\alpha_{|\Gamma_{ij}}$ and $\alpha_{|\Gamma_{ji}}$.  An evident, but important relation following directly from this definition is that
\begin{equation}
\label{eq:coeffeqv}
  \bar{a}_{\Gamma_{ij}}(u,v)\simeq a_{\Gamma_{ij}}(u,v)+a_{\Gamma_{ji}}(u,v). 
\end{equation}

The geometric domain decomposition induces a decomposition
of $V_h^0(\Omega)$ into local subspaces: for each $\Omega_i'$,
the corresponding local subspace is 
$$
  V_i:=\left\{v\in V_h^0(\Omega)\, : \,\mathrm{supp}(v)
    \subset\Omega_i'\right\}.
$$ 
This yields a decomposition of $V_h^0(\Omega)$,
$$
  V_h^0(\Omega)=V_0+\sum_{i=1}^N V_i,
$$
where the coarse space, $V_0$, is a special space which we will 
define later. For $i=0,\ldots,N$ we define the projection like
operators $T_i\colon V_h^0(\Omega)\rightarrow V_i$ as
\begin{equation}
   a(T_iu,v):=a(u,v) \qquad\forall v\in V_i,
\end{equation}
and introduce the operator
\begin{equation} \label{eq:precop}
 T:=T_0+T_1+\cdots+T_N.
\end{equation}
This allows us to replace the original problem
(\ref{eq:discreteformulation}) by the equation
\begin{equation}
\label{eq:prcndsystem}
  T u=g,
\end{equation}
where $g=\sum_{i=0}^Ng_i$ and $g_i=T_iu$. Note that $g_i$ can be
computed without knowing the solution $u$ of
(\ref{eq:discreteformulation}).

\subsection{Eigenvalue Problems on Interfaces}

The first new harmonic enrichment of our coarse space is based on
solutions of local eigenvalue problems along the interfaces between
subdomains:
\begin{df}[Generalized Eigenvalue Problem] 
For each interface $\Gamma_{ij}\subset \Gamma$, we define the generalized
eigenvalue problem: Find $\psi$ and $\lambda$, such that
\begin{equation}\label{eq:geneig}
  \bar{a}_{\Gamma_{ij}}(\psi,v):=\lambda b_{\Gamma_{ij}}(\psi,v)\qquad \forall
    v\in V_h^0(\Gamma_{ij}),
\end{equation}
where
$b_{\Gamma_{ij}}(\psi,v):=h_i^{-1}\sum\limits_{k\in\Gamma_{ijh}}\beta_{k}\psi_k v_k$ and $\beta_{k}=\sum\limits_{\substack{K\in\mathcal{T}_h(\Omega)\\k\in \mathrm{dof}(K)}}\alpha_K$.
\label{df:geneig}
\end{df}
The following lemma is a slight modification of Lemma 2.11 in \cite{Spillane:2014:GENEO} and provides important estimates for the eigenfunction projection.
\begin{lem}
\label{lem:geneig}
Define $M:=\mathrm{dim}(V_h^0(\Gamma_{ij}))$ and let the eigenpairs $\{(\psi_{\Gamma_{ij}}^k,\lambda_{ij}^k)\}_{k=1}^{\mathrm{dim}(V_h^0(\Gamma_{ij}))}$ of the generalized eigenvalue problem (\ref{eq:geneig}) be ordered such that 
\begin{equation*}
 0<\lambda_{ij}^1\leq\lambda_{ij}^2\ldots\leq\lambda_{ij}^M\quad \text{and}\quad b_{\Gamma_{ij}}(\psi_{\Gamma_{ij}}^k,\psi_{\Gamma_{ij}}^l)=\delta_{kl}\quad \text{for any}\quad 1\leq k,l\leq M.
\end{equation*}
Then, the projection for any integer $0\leq m_{ij}\leq M$ 
\begin{equation*}
 \Pi_\mathrm{m}v:=\sum_{k=1}^{m_{ij}} b_{\Gamma_{ij}}(v,\psi^k_{\Gamma_{ij}})\psi^k_{\Gamma_{ij}}
\end{equation*}
is $a$-orthogonal, and therefore
\begin{equation}
\label{eq:aorth}
 |\Pi_\mathrm{m}v|_{\bar{a}_{\Gamma_{ij}}}\leq|v|_{\bar{a}_{\Gamma_{ij}}}\quad\text{and}\quad|v-\Pi_\mathrm{m}v|_{\bar{a}_{\Gamma_{ij}}}\leq|v|_{\bar{a}_{\Gamma_{ij}}},\qquad\forall v\in V_h^0(\Gamma_{ij}).
\end{equation}
In addition we also have the approximation estimate 
\begin{equation}
\label{eq:approxeig}
 \|v-\Pi_\mathrm{m}v\|^2_{b_{\Gamma_{ij}}}\leq\frac{1}{\lambda^{m_{ij}+1}_{ij}}|v-\Pi_\mathrm{m}v|^2_{\bar{a}_{\Gamma_{ij}}},\qquad\forall v\in V_h^0(\Gamma_{ij}).
\end{equation}

\end{lem}
\begin{proof}
 Following the lines of the proof given in \cite[Lemma~2.11]{Spillane:2014:GENEO}, we start out by recognizing that since both $\bar{a}_{\Gamma_{ij}}$ and $b_{\Gamma_{ij}}$ are positive definite, we may reduce the generalized eigenvalue problem to a standard eigenvalue problem, where standard spectral theory guarantees the existence of eigenpairs $\{(\psi_{\Gamma_{ij}}^k,\lambda_{ij}^k)\}_{k=1}^{\mathrm{dim}(V_h^0(\Gamma_{ij}))}$, for which the eigenvalues $\{\lambda_{ij}^k\}_{k=1}^{\mathrm{dim}(V_h^0(\Gamma_{ij}))}$ are positive. In addition, we may choose the eigenvectors such that they form a basis of $V_h^0(\Gamma_{ij})$ and satisfy the orthogonality conditions
 $$
 \bar{a}_{\Gamma_{ij}}(\psi_{\Gamma_{ij}}^k,\psi_{\Gamma_{ij}}^l)=b_{\Gamma_{ij}}(\psi_{\Gamma_{ij}}^k,\psi_{\Gamma_{ij}}^l)=0\quad\forall k\neq l,\quad |\psi_{\Gamma_{ij}}^k|_{\bar{a}_{\Gamma_{ij}}}^2=\lambda^k_{ij}\text{ and } \|\psi_{\Gamma_{ij}}^k\|_{b_{\Gamma_{ij}}}^2=1.
 $$
Now, any $v\in V_h^0(\Gamma_{ij}))$ may be expressed as
$$
v=\sum_{k=1}^{\mathrm{dim}(V_h^0(\Gamma_{ij}))} b_{\Gamma_{ij}}(v,\psi^k_{\Gamma_{ij}})\psi^k_{\Gamma_{ij}}.
$$
The $\bar{a}_{\Gamma_{ij}}$-orthogonality states that 
$$
\left|\sum b_{\Gamma_{ij}}(v,\psi^k_{\Gamma_{ij}})\psi^k_{\Gamma_{ij}}\right|_{\bar{a}_{\Gamma_{ij}}}^2 =\sum b_{\Gamma_{ij}}(v,\psi^k_{\Gamma_{ij}})^2\left|\psi^k_{\Gamma_{ij}}\right|_{\bar{a}_{\Gamma_{ij}}}^2.
$$
From this we have 
$$
|v|_{\bar{a}_{\Gamma_{ij}}}^2=|\Pi_\mathrm{m}v|_{\bar{a}_{\Gamma_{ij}}}^2+|v-\Pi_\mathrm{m}v|_{\bar{a}_{\Gamma_{ij}}}^2,
$$
and (\ref{eq:aorth}) follows directly. To prove (\ref{eq:approxeig}), we start by using the $b_{\Gamma_{ij}}$ - orthonormality of the eigenfunctions 
\begin{eqnarray*}
 \|v-\Pi_\mathrm{m}v\|_{b_{\Gamma_{ij}}}^2&=&\left\|\sum_{k=m_{ij}+1}^{\mathrm{dim}(V_h^0(\Gamma_{ij}))} b_{\Gamma_{ij}}(v,\psi^k_{\Gamma_{ij}})\psi^k_{\Gamma_{ij}}\right\|_{b_{\Gamma_{ij}}}^2\\
 &=&\sum_{k=m_{ij}+1}^{\mathrm{dim}(V_h^0(\Gamma_{ij}))} b_{\Gamma_{ij}}(v,\psi^k_{\Gamma_{ij}})^2\frac{1}{\lambda^k_{ij}}|\psi_{\Gamma_{ij}}^k|_{\bar{a}_{\Gamma_{ij}}}^2\\
 &\leq&\frac{1}{\lambda^{m_{ij}+1}_{ij}}\sum_{k=m_{ij}+1}^{\mathrm{dim}(V_h^0(\Gamma_{ij}))} b_{\Gamma_{ij}}(v,\psi^k_{\Gamma_{ij}})^2|\psi_{\Gamma_{ij}}^k|_{\bar{a}_{\Gamma_{ij}}}^2\\
 &=&\frac{1}{\lambda^{m_{ij}+1}_{ij}}|v-\Pi_\mathrm{m}v|^2_{\bar{a}_{\Gamma_{ij}}}.
\end{eqnarray*}
In the last steps we used the fact that $|\psi_{\Gamma_{ij}}^k|_{\bar{a}_{\Gamma_{ij}}}^2=\lambda^k_{ij}$, that the eigenvalues are ordered in increasing order and that the eigenfunctions are $\bar{a}_{\Gamma_{ij}}$-orthogonal.
\end{proof}

\section{Construction of the Coarse Space}

In this section we define the coarse space for our method. The coarse
space is a multiscale finite element space enriched with harmonic
functions based on the generalized eigenvalue problem
(\ref{eq:geneig}) defined on each interface shared by two subdomains.
In order to explain the construction in detail, we start by introducing
discrete harmonic functions before we describe how the coarse space is constructed.

\subsection{Discrete Harmonic Functions}

For each non-overlapping subdomain $\Omega_i$ we define the
restriction of $V_h$ to $\bar\Omega_i$ as
$$
  V_h(\Omega_i):=\left\{v_{|\bar\Omega_i}: v\in V_h\right\},
$$
and the corresponding subspace with zero Dirichlet boundary conditions as
$$
  V_h^0(\Omega_i):=\left\{v\in V_h(\Omega_i):v(x)=0 \quad \mbox{for} \quad
x\in\partial\Omega_{ih}\right\}.
$$ 
Clearly $V_h^0(\Omega_i)\subset V_h(\Omega_i)$. Now let $\mathcal{P}_i: V_h(\Omega_i)\rightarrow V_h^0(\Omega_i)$
be the $a$-orthogonal projection of a function $u\in V_h$ onto $V_h^0(\Omega_i)$
defined by
\begin{equation}
 a_i(\mathcal{P}_iu,v):=a_i(u,v)\qquad \forall v\in V_h^0(\Omega_i),
\end{equation}
and define $\mathcal{H}_iu:=u-\mathcal{P}_iu$ as the discrete harmonic counterpart of $u$,
i.e.
\begin{eqnarray}
\label{eq:discreteharmonic}
  a_i (\mathcal{H}_i u,v)&=&0 \qquad  \forall v \in V_h^0(\Omega_i),\\
  \mathcal{H}_iu(x)&=&u(x) \qquad x \in \partial\Omega_{ih}.
\end{eqnarray}
A function $u\in V_h(\Omega_i)$ is locally discrete harmonic if $\mathcal{H}_iu=u$. If all
restrictions to subdomains of a function $u\in V_h$ are locally
discrete harmonics, i.e.,
$$
  u_{|\Omega_i}=\mathcal{H}_i u_{|\Omega_i}\qquad \text{for } i=1,\ldots,N,
$$ 
then we say $u$ is a discrete harmonic function. For any function $u\in V_h$, this gives a decomposition of $u$ into
discrete harmonic parts and local projections, i.e. $u=\mathcal{H}u+\mathcal{P}u$. 

With the operator $\mathcal{H}$, we can introduce the space of discrete harmonic functions as $$\tilde V_h=\mathcal{H}V_h=\{u \in V_h^0(\Omega): u_{|\Omega_i}=\mathcal{H}_i u_{|\Omega_i}\}.$$ For a two level method that considers a non-overlapping partitioning of $\Omega$ as subdomains, this space is the optimal coarse space. We will in the following sections show how one can approximate this space in a systematic manner by carefully including into the coarse space the most important components of $\tilde V_h$ not already included in our initial multiscale finite element coarse space. In the limit, these enrichment strategies will reach the optimal coarse space (OHEM) $\tilde V_h$, and allow us to turn the method into a direct solver. This is an important property of our strategy. 

With this in mind, we are now ready to define the coarse basis functions.

\subsection{Multiscale Basis Functions}
The multiscale finite element coarse space, which we will consider as our initial coarse space for our method, consists of vertex based discrete harmonic functions associated with each vertex of the polytope $\Omega_i$. More formally, for each vertex $x^k$ of the polytope $\Omega_i$ and for each internal edge $\Gamma_{ij} \subset \Omega_i$ that touches $x^k$, we solve the 1D boundary value problem
\begin{equation*}
  \begin{array}{rclll}
     \bar{a}_{\Gamma_{ij}}(\phi_{ik}^{ms},v)&=&0
      \quad&\forall v\in V_h^0(\Gamma_{ij}),& \mbox{(interface values)}\\
\phi_{ik}^{ms}(x^k)&=&1& \text{at the vertex $x^k$ of $\Omega_i$},\\
\phi_{ik}^{ms}&=&0& \text{at the other vertex}. 
  \end{array}
\end{equation*}
Then we extend $\phi_{ik}^{ms}$ inside $\Omega_i$ using the values of $\phi_{ik}^{ms}$ as boundary conditions on the edges touching $x^k$ and zero on the remaining edges
\begin{equation} \label{eq:discreteharmonicmultiscale}
  \begin{array}{rclll}
  a_i (\phi_{ik}^{ms},v)&=&0& \forall v \in V_h^0(\Omega_i),&\mbox{(harmonic extension inside)}\\
   \phi_{ik}^{ms}(x)&=&0& x\in\partial\Omega_i\setminus\Gamma_{ij},\quad x^k\in\partial\Gamma_{ij}.&\\
  \end{array}
\end{equation}
The solutions corresponding to each $k$ are then glued together in a natural manner to form multiscale hat functions and extended with zero to the rest of $\Omega$.
\subsection{Spectral Basis Functions} 
\begin{figure}
  \centering
  \includegraphics[width=0.5\textwidth]{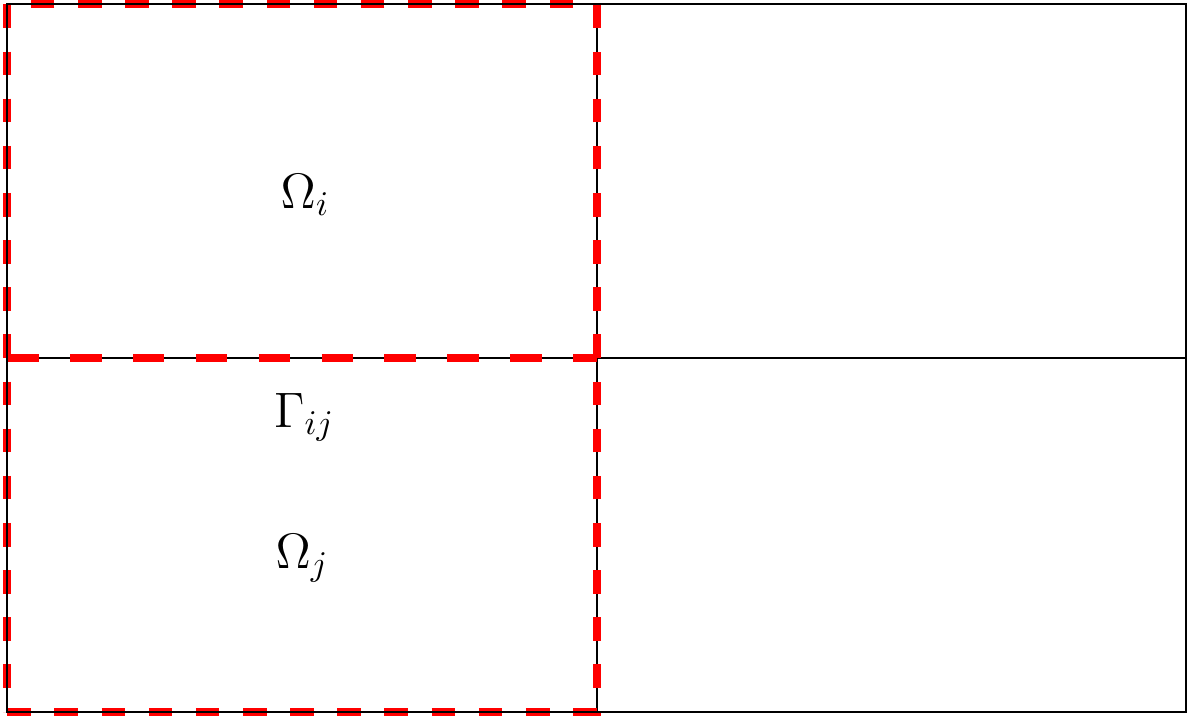}
  \caption{Support of a spectral basis function $p^k_{\Gamma_{ij}}$
    corresponding to the interface $\Gamma_{ij}$.}
\label{fig:geneiginterface}
\end{figure}
Let $\psi^k_{\Gamma_{ij}}$ be the $k$-th eigenfunction of the generalized
eigenvalue problem (\ref{eq:geneig}) on each interface
$\Gamma_{ij}$. We then define the spectral basis functions associated
with the interface $\Gamma_{ij}$ as the discrete harmonic extension of
each of the eigenfunctions $\psi^k_{\Gamma_{ij}}$
\begin{equation}\label{eq:discreteharmonicspectral}
  \begin{array}{rcll}
  a_i (p^k_{\Gamma_{ij}},v)&=&0 \quad&  \forall v \in V_h^0(\Omega_i),\\
  p^k_{\Gamma_{ij}}&=&\psi^k_{\Gamma_{ij}}& x\in\Gamma_{ij},\\
p^k_{\Gamma_{ij}}&=&0& x\in\left(\partial\Omega_i\cup\partial\Omega_j\right)\setminus\Gamma_{ij}.
  \end{array}
\end{equation}
Each solution is then extended with zero to the rest of $\Omega$ in order to make the functions global.

The spectral harmonically enriched multiscale coarse space (SHEM) is
then defined as the span of these two sets of basis functions,
\begin{equation}\label{SHEM}
  V_0:=\mathrm{span}\{\{\phi_k^{ms}\}_{k=1}^{n_\nu}\cup\{\{p^l_{\Gamma_{ij}}\}_{l=1}^{m_{ij}}\}_{\Gamma_{ij}\subset\Gamma}\}.
\end{equation}
An example of the first and second order spectral basis function is given in Figure~\ref{figs:spectral} for a problem with $\alpha=1$ and mesh parameters $H=1/2$ and $h=1/32$.
\begin{figure}
\centering
\mbox{\includegraphics[width=0.5\textwidth]{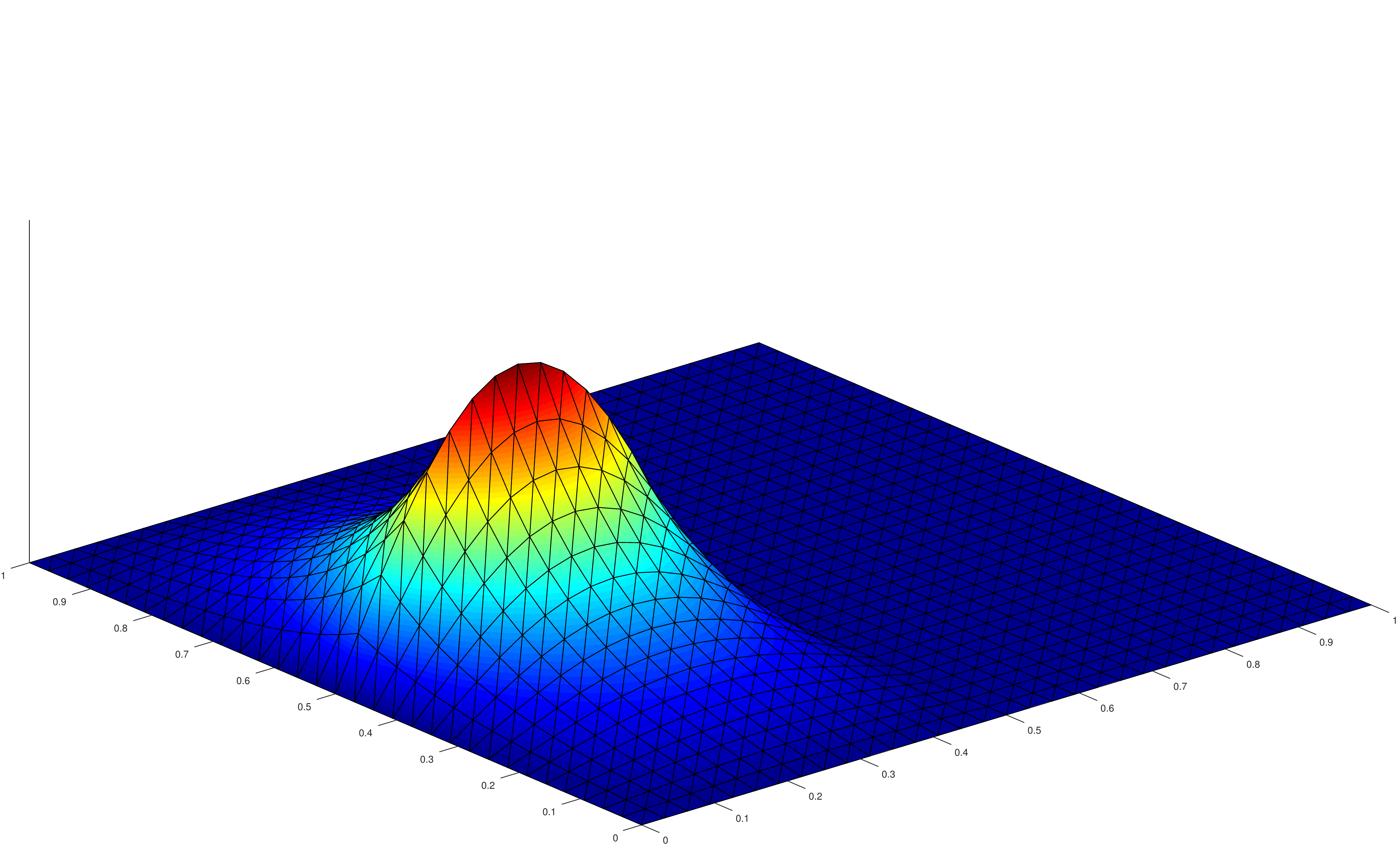}
\includegraphics[width=0.5\textwidth]{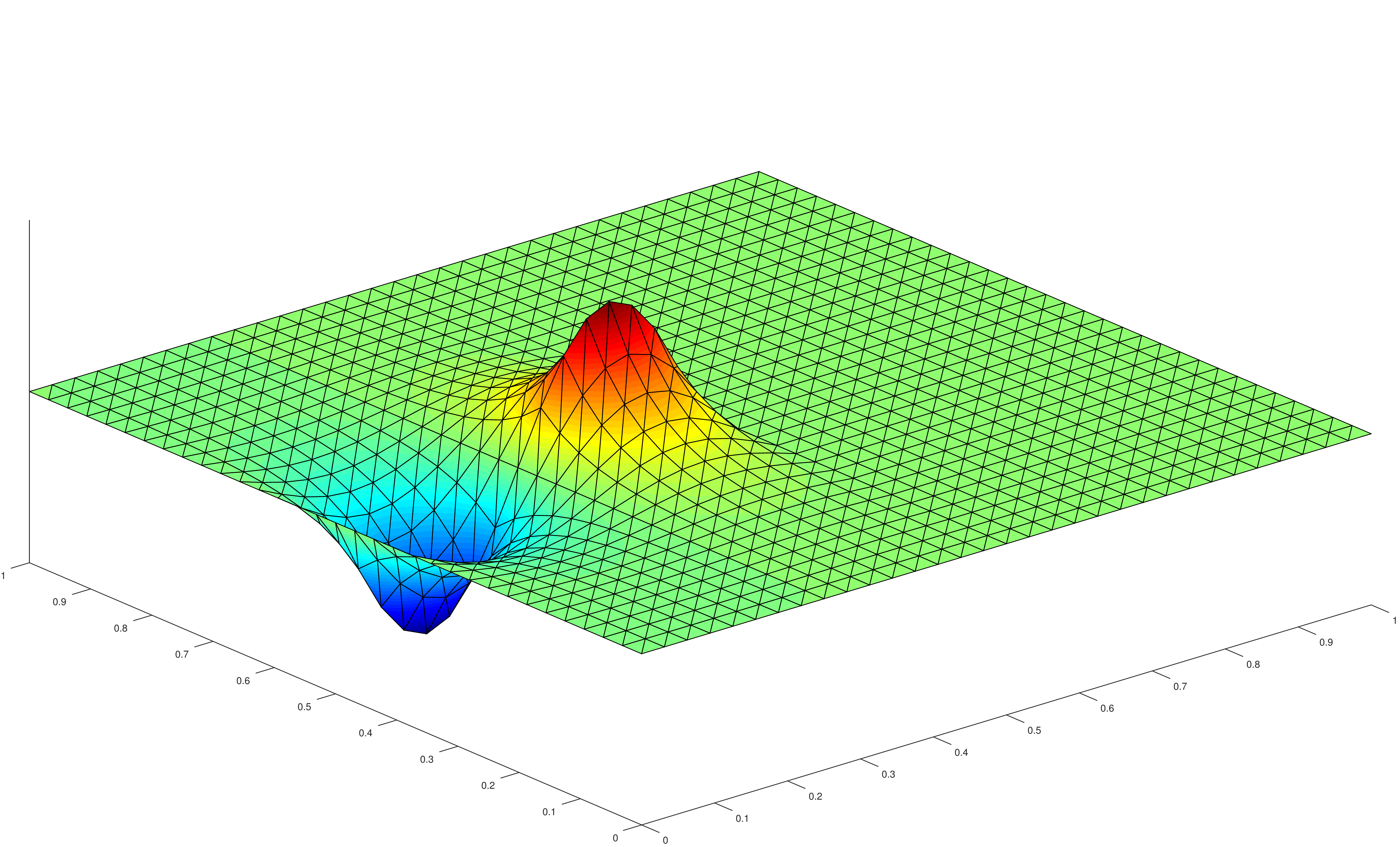}
}
\caption{Plot of the spectral basis functions for $h=1/32$ and $H=16h$. \textit{Left:} first order spectral basis functions. \textit{Right:} second order spectral basis functions.}
\label{figs:spectral}
\end{figure}

\section{Analysis of the Preconditioner}

In this section we define a suitable interpolation operator into the
SHEM coarse space $V_0$ and provide the necessary bounds for the operator
in order to apply the classical abstract Schwarz
framework. The aforementioned interpolation operator is a composition
of the standard multiscale interpolation operator and a new
interpolation operator for the spectral basis functions.

The multiscale interpolation operator is defined as
\begin{equation}\label{eq:multint}
   I_\mathrm{ms}u:=\sum_{k=1}^{n_\nu} u(x^k)\phi_k^\mathrm{ms},
\end{equation}
where $n_\nu$ is the number of internal vertices, i.e. vertices of the polytope $\Omega_i$ which are not on $\partial\Omega$. The spectral interpolation operator is defined as 
\begin{equation}
 \label{eq:specint}
\Pi_\mathrm{m}u:=\sum_{\Gamma_{ij}\subset\Gamma}\sum_{k=1}^{m_{ij}} b_{\Gamma_{ij}}(u_{\Gamma_{ij}},p^k_{\Gamma_{ij}})p^k_{\Gamma_{ij}}.
\end{equation}
Combining these two, the new coarse space interpolation operator $I_0\colon V_h^0(\Omega)\rightarrow V_0$ is defined as
\begin{equation}
 \label{eq:coarseint}
I_0u:=I_\mathrm{ms}u+\Pi_\mathrm{m}\left(u-I_\mathrm{ms}u\right).
\end{equation}

\begin{lem}[Stable Decomposition]
\label{lem:stbldec}
 For all $u\in V_h$ there exists a representation $u=\sum_{i=0}^Nu_i$ such that 
 \begin{equation}
 \label{eq:stldec}
  a(u_0,u_0)+\sum_{i=1}^N a(u_i,u_i)\preceq C_0^2 a(u,u),
 \end{equation}
where $u_0=I_0u$, $C_0^2\simeq \left(1+\frac{1}{\lambda_{m+1}}\right)$ and $\lambda_{m+1}:=\min\limits_{i}\min\limits_{\Gamma_{ij}\subset\partial\Omega_i}\lambda_{m_{ij}+1}^{ij}$.
\end{lem}
\begin{proof}
Let $w:=u-u_0$ and define $u_i:=I_h(\theta_iw)$, where $I_h$ is the nodal piecewise linear continuous interpolation operator on the fine mesh $\mathcal{T}_ h(\Omega)$ and $\theta_i$ is a partition of unity function with respect to the partition $\{\Omega_i'\}_{i=1}^N$. $\theta_i$ is zero on $\Omega\setminus\Omega_i'$ and $\theta(x)=1$ for $x\in\Omega_{ih}$, $\theta_i(x)=1/2$ for $x\in\Gamma_{ijh}$, and $\theta_i(x)=1$ when $x=x^k$. 

We start by estimating
$a(u_0,u_0)=\sum_{i=1}^Na_i(u_0,u_0)$. From the energy minimizing
property of discrete harmonic functions we have for all $v\in
V_h(\Omega_i)$ where $v=u_0$ on the boundary of $\partial\Omega_i$ and
$v=u$ on $\Omega_{ih}$ that
\begin{eqnarray*}
  a_i(u_0,u_0)&\leq& a_i(v,v)= a_i(v-u+u,v-u+u)\\
  &\leq&2a_i(v-u,v-u)+2a_i(u,u).
\end{eqnarray*}
Using the fact that $u-v$ is zero on $\Omega_{ih}$ we then get 
\begin{equation}
 a_i(v-u,v-u)\preceq\sum_{\Gamma_{ij}\subset\partial\Omega_i}h_i\|u-u_0\|^2_{b_{\Gamma_{ij}}}.
\end{equation}
By applying Lemma~\ref{lem:geneig} we have for each edge
$\Gamma_{ij}\subset\partial\Omega_i$ that
\begin{eqnarray}
 \|u-u_0\|^2_{b_{\Gamma_{ij}}}&=&\|(u-I_\mathrm{ms}u)-\Pi_\mathrm{m}(u-I_\mathrm{ms}u)\|^2_{b_{\Gamma_{ij}}}\nonumber\\
 &\preceq&\frac{1}{\lambda^{m_{ij}+1}_{ij}}|(u-I_\mathrm{ms}u)-\Pi_\mathrm{m}(u-I_\mathrm{ms}u)|_{\bar{a}_{\Gamma_{ij}}}^2\nonumber\\
 &\preceq& \frac{1}{\lambda^{m_{ij}+1}_{ij}}|u-I_\mathrm{ms}u|_{\bar{a}_{\Gamma_{ij}}}^2\preceq \frac{1}{\lambda^{m_{ij}+1}_{ij}}|u|_{\bar{a}_{\Gamma_{ij}}}^2.
\end{eqnarray}
The last inequality follows straightforwardly from Lemma~\ref{lem:geneig} and the $a$-stability property of the 1D multiscale operator. By defining  $\lambda_{m_i+1}^i:=\min\limits_{\Gamma_{ij}\subset\partial\Omega_i}\lambda_{m_{ij}+1}^{ij}$ and using (\ref{eq:coeffeqv}) we have that
\begin{eqnarray}
 \sum_{\Gamma_{ij}\subset\partial\Omega_i}\frac{h_i}{\lambda^{m_{ij}+1}_{ij}}|u|_{\bar{a}_{\Gamma_{ij}}}^2&\simeq&\sum_{\Gamma_{ij}\subset\partial\Omega_i}\frac{h_i}{\lambda^{m_{ij}+1}_{ij}}\left(|u|_{a_{\Gamma_{ij}}}^2+|u|_{a_{\Gamma_{ji}}}^2\right)\nonumber\\
 &\preceq&\frac{1}{\lambda_{m_i+1}^{i}}\left(|u|_{a,\Omega_i^h}^2+\sum_{\Gamma_{ij}\subset\partial\Omega_i}|u|_{a,\Omega_j^h}^2\right)\nonumber\\
 &\leq&\frac{1}{\lambda_{m_i+1}^{i}}\left(|u|_{a,\Omega_i}^2+\sum_{\Gamma_{ij}\subset\partial\Omega_i}|u|_{a,\Omega_j}^2\right),
\end{eqnarray}
where in the second inequality above we extend the estimate from the boundary to the subdomain layer, while in the last inequality we extend from the subdomain layer to the whole of the subdomain. Then, by summing over each subdomain and defining $\lambda_{m+1}:=\min\limits_{i}\lambda_{m_i+1}^{i}$ completes the first part of the proof.

Now we need to provide the same type of bound for the local terms $u_i$, i.e. 
\begin{equation}
 \sum_{i=1}^Na(u_i,u_i)\preceq \frac{1}{\lambda_{m+1}}a(u,u).
\end{equation}
Since we are only considering two layers of overlap and by adding and subtracting $u-u_0$ we have
\begin{eqnarray}
a(u_i,u_i)&=&a_{\Omega_i'\setminus\overline\Omega_i}(u_i,u_i)+a_{\Omega_i}(u_i-(u-u_0)+(u-u_0),u_i-(u-u_0)+(u-u_0))\nonumber\\ 
&\leq&a_{\Omega_i'\setminus\overline\Omega_i}(u_i,u_i)+ 2a_{\Omega_i}(u-u_0,u-u_0)+2a_{\Omega_i}(u_i-(u-u_0),u_i-(u-u_0))\nonumber\\ 
 &\simeq& a_{\Omega_i}(u-u_0,u-u_0)+\sum_{\Gamma_{ij}\subset\partial\Omega_i}h_i\|u-u_0\|^2_{b_{\Gamma_{ij}}}\nonumber.
\end{eqnarray}
The last term in the inequality above has already been estimated so finally we have that
\begin{equation}
 \sum_{i=0}^Na(u_i,u_i)\preceq \left(1+\frac{1}{\lambda_{m+1}}\right)a(u,u),
\end{equation}
which completes the proof of (\ref{eq:stldec}).

\end{proof}
\begin{remark}
The theoretical results developed in this section easily extend to the case of minimal overlap $h$. The only modifications needed in the proof of Lemma~\ref{lem:stbldec} are for the local components $a(u_i,u_i)$. Instead of splitting the overlapping subdomain $\Omega_i'$ into $\Omega_i'\setminus\overline\Omega_i$ and $\Omega_i$, one instead splits the overlapping zone of $\Omega_i'$, i.e., the part that is shared by two or more subdomains, into the part outside of the coarse grid elements and the part that is inside the coarse grid elements and then proceed in a similar fashion as in the proof above. 
\end{remark}
\begin{thm}
\label{thm:mainthm}
 The condition number of the two level Schwarz operator (\ref{eq:precop}) with the SHEM coarse space (\ref{SHEM}) can be bounded by 
 \begin{equation}
  \kappa(\mathbf{TA})\preceq \omega C_0^2(\rho(E)+1)
  \end{equation}
  where $C_0$ is defined as in Lemma~\ref{lem:stbldec}, $\omega=1$ and $\mathbf{TA}$ is the matrix representation of our preconditioned system.

\end{thm}
\begin{proof}
 Following the Schwarz framework, cf. \cite{Toselli:2005:DDM,smith1996domain}, we need to prove three assumptions:
\begin{ass}[1]
 This assumption is the stable decomposition one which we have already proved in Lemma~\ref{lem:stbldec}.
\end{ass}
\begin{ass}[2] Let $0\leq\mathcal{E}_{ij}\leq1$ be the minimal values that satisfy 
\begin{equation*}
 a(u_i,u_j)\leq\mathcal{E}_{ij}a(u_i,u_i)^{1/2}a(u_j,u_j)^{1/2},\qquad\forall u_i\in V_i,\;\forall u_j\in V_j\;i,j=1,\ldots,N
\end{equation*}
Define $\rho(\mathcal{E})$ to be the spectral radius of $\mathcal{E}=\{\mathcal{E}_{ij}\}$.
\end{ass}

\begin{ass}[3]
Let $\omega>0$ be the minimal constant such that 
\begin{equation*}
 a(u,u)\leq\omega b_i(u,u),\qquad u\in V_i.
\end{equation*}
\end{ass}
We use exact bilinear forms, i.e., $b_i(u,u)=a(u,u)$, so in our case $\omega=1$ for $i=1,\ldots,N$.
\end{proof}

\section{Non-Spectral Harmonic Enrichment}
\label{sec:nonspec}
Similar coarse basis functions as the ones constructed in the previous section may also be constructed without solving eigenvalue problems on the interfaces. Instead we may solve local lower dimensional problems and extend the solutions harmonically inside each subdomain in the same manner as for the eigenfunctions. 

For this variant of the method we construct the basis functions by solving 
on each interface $\Gamma_{ij}$ the 1D problem 
 \begin{eqnarray}
\label{eq:bubbleconstruction}
  \bar{a}_{\Gamma_{ij}}(\phi^k_{\Gamma_{ij}},v)&=&b_{\Gamma_{ij}}(g^k,v)\qquad\forall v\in V_h^0(\Gamma_{ij}),\\
\phi^k_{\Gamma_{ij}}&=&0\qquad x\in\partial\Gamma_{ij},\nonumber
\end{eqnarray}
where $b_{\Gamma_{ij}}(\cdot,\cdot)$ is given in Definition~\ref{df:geneig} and the alternating function $g^k$ is defined in the following way: let $\mathbf{\gamma}_{ij}:[0,1]\rightarrow\Gamma_{ij}$ be a parametrization of the interface $\Gamma_{ij}$, where $\mathbf{\gamma}_{ij}(0)$ and $\mathbf{\gamma}_{ij}(1)$ describe the end points of $\Gamma_{ij}$ and let $k$ denote the number of the basis functions used for enrichment. The alternating function $g^k(x)$ is then defined for each $k$ by
$$g^k(\mathbf{\gamma}_{ij}(t)):=\begin{cases} \;\;\,1, & t\in[0,\frac{1}{k}], \\-1, &t\in(\frac{1}{k},\frac{2}{k}],\\\;\;\;\vdots\\(-1)^{k-1}, &t\in(\frac{k-1}{k},1].\end{cases}$$
Other choices for $g^k$ are also possible: one could for instance choose $g^k$ to be a family of sine functions or hierarchical basis functions. In any case the crucial part for achieving the same robust behavior as in the eigenfunction case is that we use the weighted $L^2$ inner product on the right hand side in (\ref{eq:bubbleconstruction}). 

Next, we extend these solutions harmonically inside the two subdomains sharing $\Gamma_{ij}$ as an edge in the same manner as for the eigenfunctions, i.e. for each subdomain sharing $\Gamma_{ij}$, we solve
\begin{equation}\label{eq:discreteharmonicbubble}
  \begin{array}{rcll}
  a_i (\chi^k_{\Gamma_{ij}},v)&=&0 \quad&  \forall v \in V_h^0(\Omega_i),\\
  \chi^k_{\Gamma_{ij}}&=&\phi^k_{\Gamma_{ij}}& x\in\Gamma_{ij},\\
\chi^k_{\Gamma_{ij}}&=&0& x\in\left(\partial\Omega_i\cup\partial\Omega_j\right)\setminus\Gamma_{ij},
  \end{array}
\end{equation}
and extend the solution with zero to the rest of $\Omega$.

The non-spectral harmonically enriched multiscale coarse space (NSHEM)
is then defined as the span of the two sets of basis functions
\begin{equation}\label{NSHEM}
  V_0^*:=\mathrm{span}\{\{\phi_k^{ms}\}_{k=1}^{n_\nu}\cup\{\{\chi^l_{\Gamma_{ij}}\}_{l=1}^{m_{ij}}\}_{\Gamma_{ij}\subset\Gamma}.
\end{equation}
We will also refer to the above basis functions as non-spectral functions. An example of such functions for $k=1$ and $k=2$ are given in Figure~\ref{figs:nonspectral} for a problem with $\alpha=1$ and mesh parameters $H=1/2$ and $h=1/32$. We see a close resemblance to the spectral basis functions given in Figure~\ref{figs:spectral} and in the next section we show with numerical examples that the behavior for these variants of the coarse enrichment is almost identical to the behavior of the eigenfunction variant. 
\begin{figure}
\centering
\mbox{\includegraphics[width=0.5\textwidth]{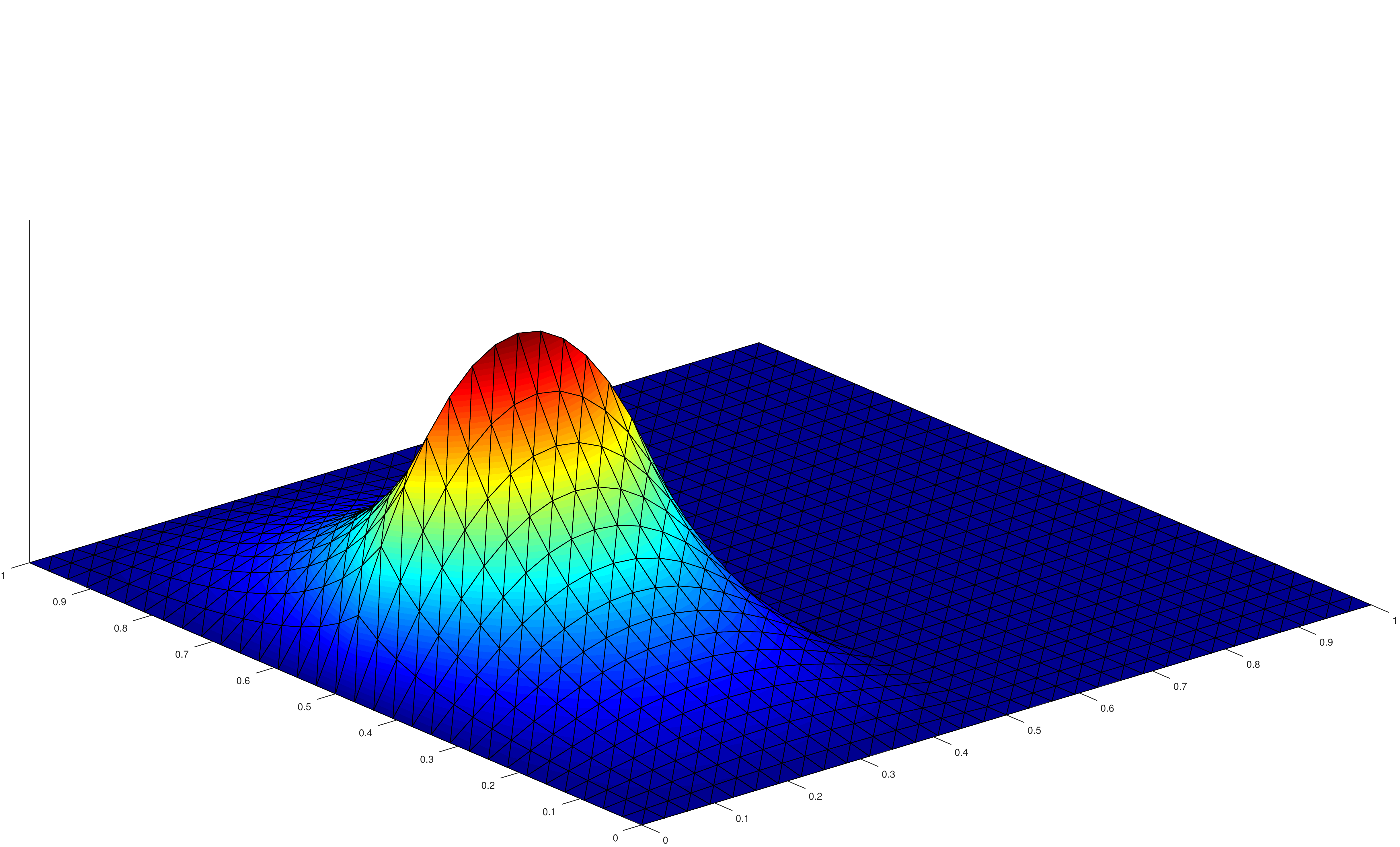}
\includegraphics[width=0.5\textwidth]{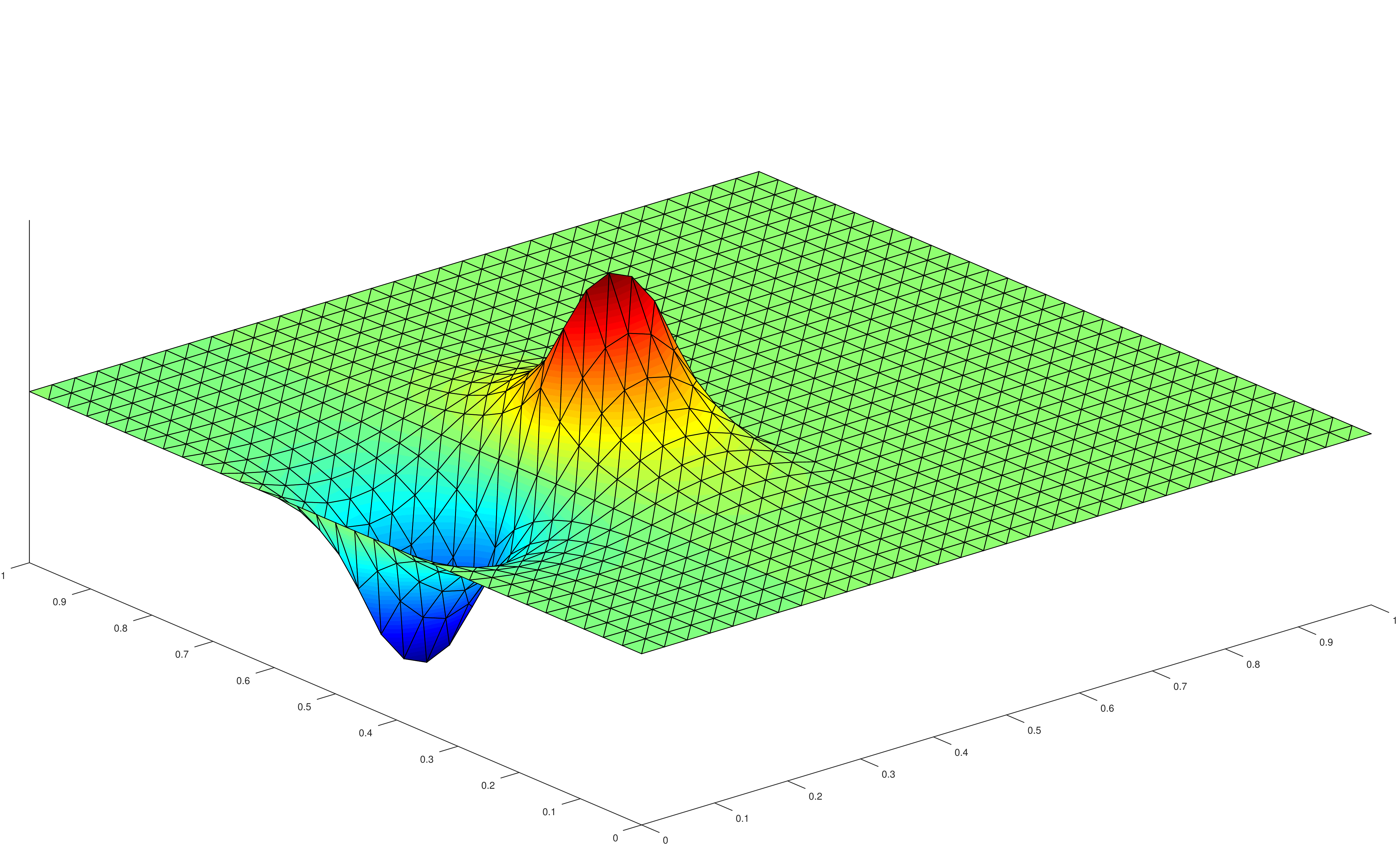}
}
\caption{Plot of the non-spectral basis functions for $h=1/32$ and $H=16h$. \textit{Left:} first order non-spectral basis functions. \textit{Right:} second order non-spectral basis functions.}
\label{figs:nonspectral}
\end{figure}

\section{Numerical Results}

We now present extensive numerical experiments for the new coarse
spaces solving problem (\ref{eq:modelproblem}) on a unit square
domain, i.e. $\Omega=(0,1)^2$, where the right hand side is chosen to
be $f=1$, and the coefficient $\alpha(x)$ represents various (possibly
discontinuous) distributions. We test the new coarse spaces with the
two level additive Schwarz (AS) method, where the coarse spaces have
been enriched with different numbers of spectral and non-spectral
basis functions on each interface, and then AS is used as a preconditioner
for the conjugate gradient method. For all numerical examples, we run
the preconditioned conjugate gradient method until the $l_2$ norm of
the initial residual is reduced by a factor of $10^{6}$, that is,
until $\|r_i\|_2/\|r_0\|_2\leq 10^{-6}$. The coefficient $\alpha(x)$
for all the numerical examples is equal to $1$, except in the areas
marked with red where the value of $\alpha(x)$ is equal to
$\hat\alpha$.

\subsection{Comparison of SHEM and NSHEM}

We start by showing that SHEM and the three variants of NSHEM
introduced in the previous sections have similar performance. To do
so, we use the distribution of $\alpha$ with three inclusions of
different size crossing interfaces between subdomains
in Figure~\ref{fig:3channels}. 
\begin{figure}
   \centering
   \includegraphics[width=\textwidth]{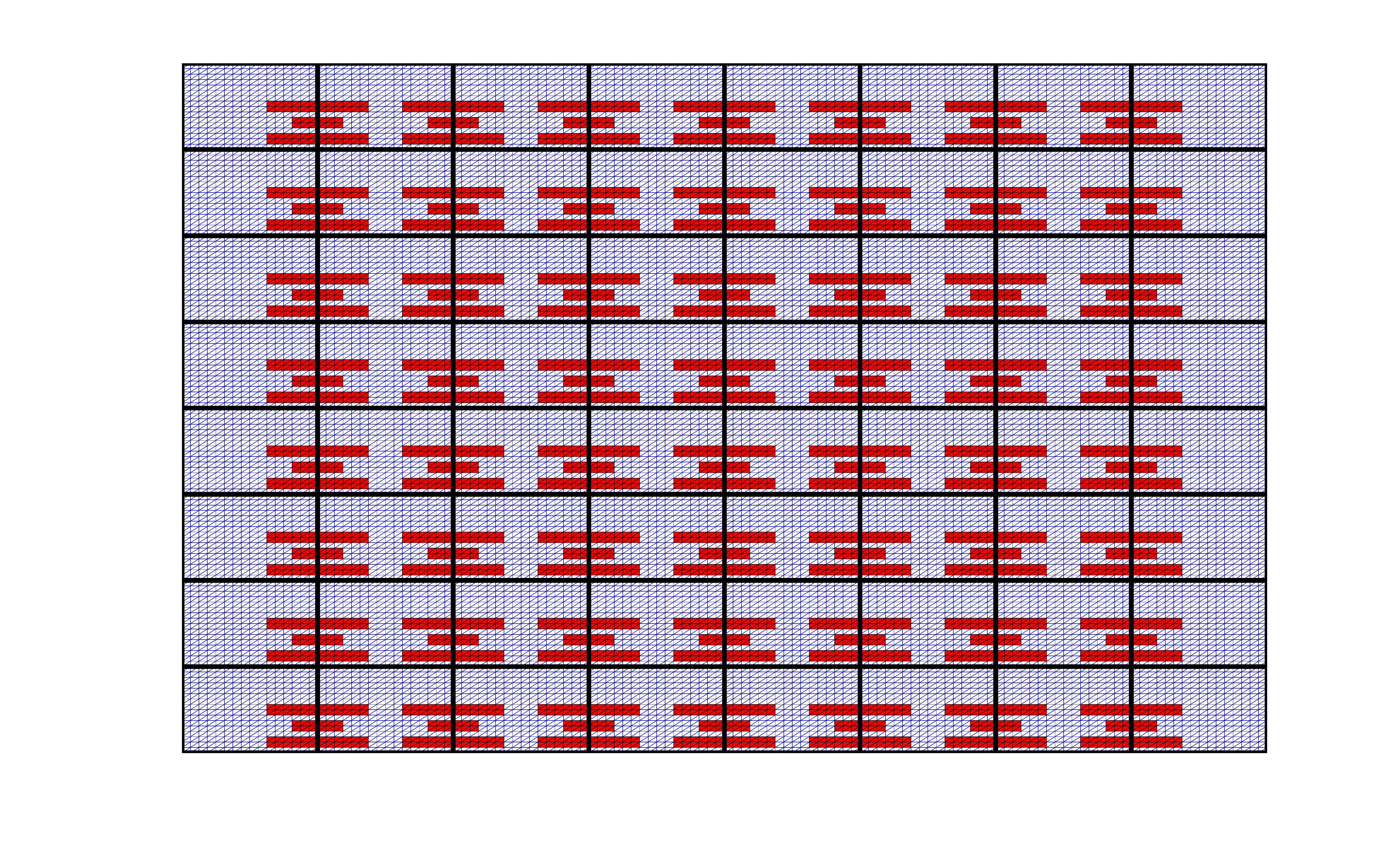}
  \caption{Distribution of $\alpha$ for a geometry with
    $h=\frac{1}{128}$, $H=16h$. The regions marked with red are where
    $\alpha$ has a large value.}
  \label{fig:3channels}
\end{figure}
We test the case where the multiscale coarse space is enriched with
three basis functions on each interface (SHEM$_3$) and compare the number of
iterations and condition number in Table~\ref{tbl:coarsecomp} to 
NSHEM$_3^a$, which denotes the variant with three piecewise constant
alternating function $g^k$ in the definition
(\ref{eq:bubbleconstruction}), NSHEM$_3^s$, which uses sine functions for
$g^k$, and NSHEM$_3^h$, which uses hierarchical basis functions for $g^k$.
\begin{table}\footnotesize
\centering
\begin{tabular}{c c c c c}
\hline
Type:&SHEM$_3$&NSHEM$_3^a$&NSHEM$_3^s$&NSHEM$_3^h$\\
\hline
$\hat\alpha$&\#it. $(\kappa)$&\#it. $(\kappa)$&\#it. $(\kappa)$&\#it. $(\kappa)$\\
\hline
  $10^0$&13 (5.19e0)&13 (5.50e0)&13 (5.19e0)&17 (8.26e0)\\
  $10^2$&18 (7.42e0)&19 (7.74e0)&19 (7.80e0)&20 (7.58e0)\\
  $10^4$&19 (7.44e0)&19 (7.49e0)&19 (7.44e0)&20 (7.60e0)\\
  $10^6$&19 (7.44e0)&19 (7.49e0)&19 (7.44e0)&21 (7.61e0)\\
\hline 
\end{tabular}
\caption{Comparison of the iteration count and condition number
  estimate for SHEM$_3$ and the three NSHEM$_3$ variants for the distribution
  given in Figure~\ref{fig:3channels} with $h=\frac{1}{128}$,
  $H=16h$.}
  \label{tbl:coarsecomp}
\end{table} 
We see from the results in Table~\ref{tbl:coarsecomp} that the performance of the four variants is very similar. In fact, for the first three variants of the coarse enrichment, the performance of the method is basically identical.

We next revisit the distribution of $\alpha$ given in Figure~\ref{fig:complicatedex}, and compare SHEM$_i$ and NSHEM$_i^a$ with a varying number of basis functions $i=1,2,3,4$ on each interface. We state the number of iterations and condition number for each run in the Table~\ref{tbl:spectralex} and \ref{tbl:bubbleex}. 
\begin{table}\footnotesize
\centering
\begin{tabular}{c c c c c c}
\hline
   &MS&SHEM$_1$&SHEM$_2$&SHEM$_3$&SHEM$_4$\\
\hline
dim.&49&161&273&385&497\\
\hline
$\hat\alpha$&\#it. $(\kappa)$&\#it. $(\kappa)$&\#it. $(\kappa)$&\#it. $(\kappa)$&\#it. $(\kappa)$\\
\hline
  $10^0$&21 (1.29e1)&16 (7.45e0)&15 (5.99e0)&13 (5.19e0)&13 (5.15e0)\\
  $10^2$&122 (3.74e2)&70 (1.17e2)&47 (6.70e1)&19 (6.77e0)&16 (5.66e0)\\
  $10^4$&367 (3.64e4)&248 (1.10e4)&187 (6.22e3)&19 (6.78e0)&17 (5.73e0)\\
  $10^6$&610 (3.64e6)&423 (1.10e6)&290 (6.22e5)&19 (6.78e0)&17 (5.73e0)\\
\hline
\end{tabular}
\caption{Iteration count and condition number estimate for the
  distribution in Figure~\ref{fig:complicatedex} for the classical
  multiscale coarse space (MS) and different numbers of spectral basis
  functions $i=1,2,3,4$ in SHEM$_i$ with $h=\frac{1}{128}$, $H=16h$.}
  \label{tbl:spectralex}
\end{table} 
\begin{table}\footnotesize
\centering
\begin{tabular}{c c c c c}
\hline
 &NSHEM$_1^a$&NSHEM$_2^a$&NSHEM$_3^a$&NSHEM$_4^a$\\
\hline
dim.&161&273&385&497\\
\hline
$\hat\alpha$&\#it. $(\kappa)$&\#it. $(\kappa)$&\#it. $(\kappa)$&\#it. $(\kappa)$\\
\hline
  $10^0$&16 (7.52e0)&14 (6.06e0)&13 (5.50e0)&13 (5.23e0)\\
  $10^2$&71 (1.17e2)&49 (6.92e1)&19 (6.79e0)&17 (5.74e0)\\
  $10^4$&256 (1.11e4)&130 (6.45e3)&20 (6.81e0)&20 (6.67e0)\\
  $10^6$&454 (1.11e6)&221 (6.45e5)&20 (6.81e0)&20 (6.67e0)\\
\hline 
\end{tabular}
\caption{Iteration count and condition number estimate for the
  distribution in Figure~\ref{fig:complicatedex} for different numbers
  of non-spectral basis functions $i=1,2,3,4$ in NSHEM$_i^a$ with
  $h=\frac{1}{128}$, $H=16h$.}
  \label{tbl:bubbleex}
\end{table} 
We see again that the performance of SHEM$_i$ and NSHEM$_i^a$ is very similar,
and the solver becomes robust for the same number $i=3$ of enrichment functions, corresponding to the number of inclusions across the interfaces.

\subsection{An Adaptive Variant of SHEM}

In Table~\ref{tbl:adaptive} 
\begin{table}\footnotesize
\centering
\begin{tabular}{c c c c}
\hline
&SHEM$_{\tiny\mbox{adapt}}$&\\
\hline
$\hat\alpha$&\#it. $(\kappa)$&dim.\\
\hline
  $10^0$&21 (1.29e1)&49\\
  $10^2$&25 (1.10e1)&233\\
  $10^4$&25 (1.09e1)&233\\
  $10^6$&25 (1.09e1)&233\\
\hline
\end{tabular}
\caption{Iteration and condition number estimate for the distribution
  in Figure~\ref{fig:complicatedex} for SHEM$_{\tiny\mbox{adapt}}$
  with $h=\frac{1}{128}$, $H=16h$.}
\label{tbl:adaptive}
\end{table}
we give the number of iterations and condition number for an adaptive
version we call SHEM$_{\tiny\mbox{adapt}}$ in the case of the distribution
of $\alpha$ given in Figure~\ref{fig:complicatedex}: on each interface
shared by two subdomains we enrich the coarse space with spectral
basis functions corresponding to eigenvalues below a certain
threshold. By comparing the eigenvalues to the eigenvalues of the
Laplacian we may choose the threshold in such a way that we only
include spectral functions corresponding to eigenvalues that are
smaller than the smallest eigenvalue of the Laplacian on each
interface. Eigenvalues below this threshold will correspond to
discontinuities along the subdomain boundaries.

By studying the dimension of the enriched coarse space in
Table~\ref{tbl:adaptive} we see that by a proper choice of the cut-off
criteria the method is completely insensitive to any discontinuity
inside subdomains and along subdomain boundaries. Also, if we count
the number of high conducting regions in
Figure~\ref{fig:complicatedex} crossing the interfaces, we see that
the number of eigenfunctions needed on each interface equals the
number of inclusions or channels crossing it.


\subsection{An Example of OHEM}

If all of the spectral basis functions corresponding to each interface
are included, then the coarse space spans the full discrete harmonic
function space $\tilde V_h$, which is the optimal harmonically enriched multiscale
coarse space (OHEM), and the method can be made into a
direct solver: The subdomain solves are only used to incorporate the
influence of the right hand side function into the solution, while the
full harmonic space then connects and shifts them correctly, and
overlap is not needed any more. In Table~\ref{tbl:allevels}, 
\begin{table}\footnotesize
\centering
\begin{tabular}{c c c}
\hline
&non-overlapping&overlapping\\
\hline
$\hat\alpha$&\#it. $(\kappa)$&\#it. $(\kappa)$\\
\hline
  $10^0$&1 (1)&10 (5)\\
  $10^2$&1 (1)&13 (5)\\
  $10^4$&1 (1)&13 (5)\\
  $10^6$&1 (1)&13 (5)\\
\hline
\end{tabular}
\caption{Iteration and condition number estimate for OHEM applied to the
  distribution in Figure~\ref{fig:complicatedex} with
  $h=\frac{1}{128}$, $H=16h$.}
\label{tbl:allevels}
\end{table}
we show some numerical results for the non-overlapping case, where we
obtain the solution in one iteration, and the corresponding
overlapping method, where we see that the error the additive Schwarz
method commits in the overlap to remain symmetric (see
\cite{Gander:2003:RAS,Gander:2008:SchwarzTime}) prevents the method
from converging in one iteration. This can be corrected using RAS
which will be shown in \cite{gander2016RASEIG}, see also
\cite{gander2013new}. OHEM is not really a practical method, since the
coarse space is very big, but it is of great theoretical interest,
since it defines precisely which object a good coarse space should
approximate, and this is how we discovered SHEM and NSHEM, see also
\cite{gander2011optimal} where this complete coarse space information
is encoded in transmission conditions between subdomains.

\subsection{Highly Irregular Conductivities and Subdomains}

We now test the case where we allow the distribution of $\alpha$ to be
highly irregular. Inspired by an example in
\cite{Willems:2013:Sepctral} we consider a slightly modified version
of it, as depicted in Figure~\ref{fig:complicatedex2}.
\begin{figure}
   \centering
   \includegraphics[width=\linewidth]{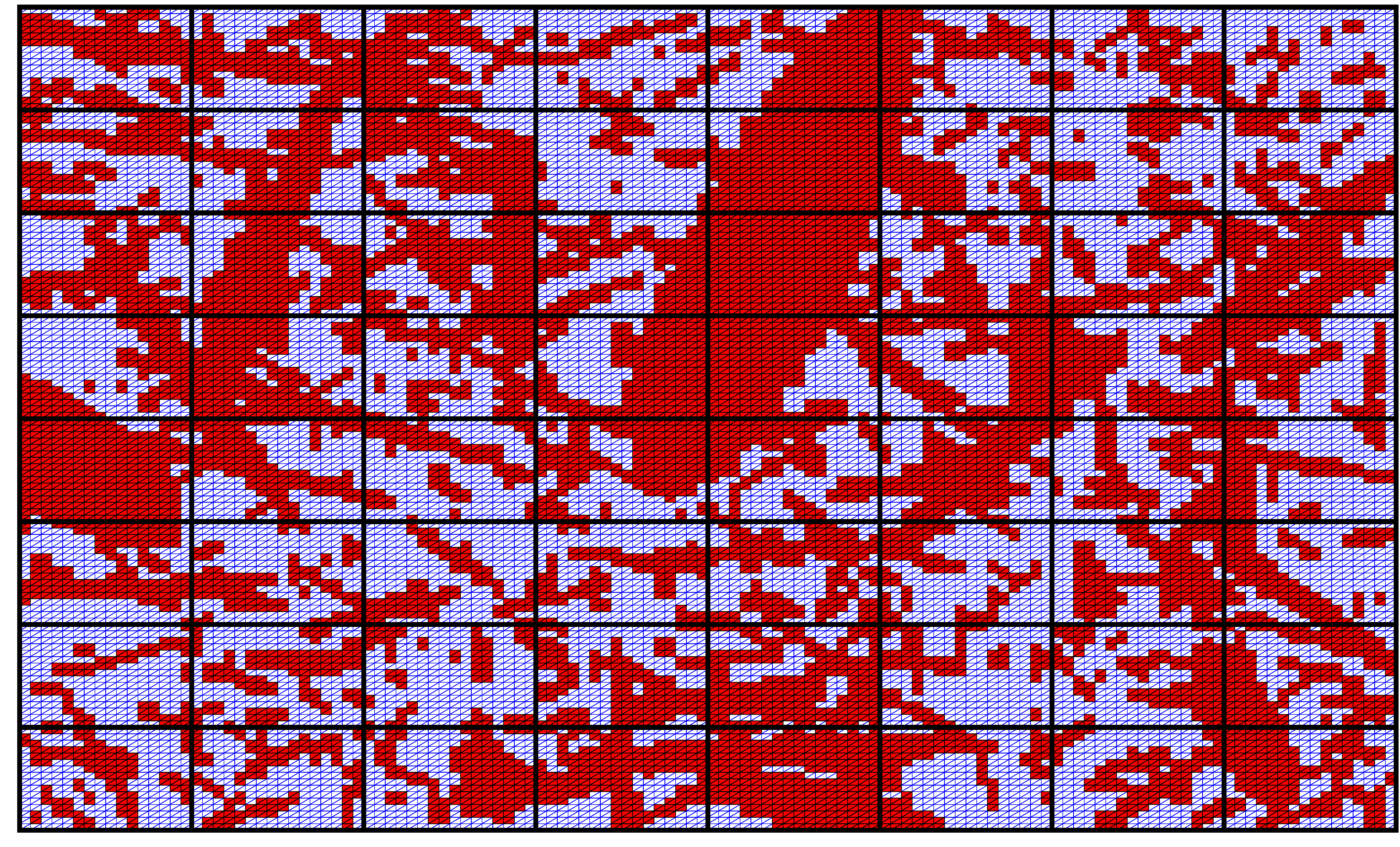}
  \caption{Distribution of $\alpha$ for a geometry with
    $h=\frac{1}{128}$, $H=16h$. The regions marked with red are where
    $\alpha$ has a large value.}
  \label{fig:complicatedex2}
\end{figure}
For this case we choose the threshold for the adaptive variant
SHEM$_{\tiny\mbox{adapt}}$ in such a way that we are guaranteed at least
one spectral basis function on each edge. The number of iterations and
condition numbers for increasing $\hat\alpha$ are given in
Table~\ref{tbl:adaptive2}
\begin{table}\footnotesize
\centering
\begin{tabular}{c c c|| c c}
\hline
Type:&\multicolumn{2}{c||}{SHEM$_3$}&\multicolumn{2}{c}{SHEM$_{\tiny\mbox{adapt}}$}\\
\hline
$\hat\alpha$&\#it. $(\kappa)$&dim.&\#it. $(\kappa)$&dim.\\
\hline
  $10^0$&13 (5.19e0)&385&16 (7.45e0)&161\\
  $10^2$&22 (9.47e0)&385&25 (1.07e1)&239\\
  $10^4$&23 (9.60e0)&385&26 (1.11e1)&236\\
  $10^6$&24 (9.59e0)&385&28 (1.08e1)&236\\
\hline
\end{tabular}
\caption{Iteration and condition number estimate for the distribution
  in Figure~\ref{fig:complicatedex2} comparing SHEM$_3$ and
  SHEM$_{\tiny\mbox{adapt}}$ with $h=\frac{1}{128}$, $H=16h$.}
\label{tbl:adaptive2}
\end{table}
for SHEM$_3$ and SHEM$_{\tiny\mbox{adapt}}$. We see from the table that the
dimension of the coarse space is still lower than the dimension for
the local subspaces. In addition, we see that the performance of SHEM$_{\tiny\mbox{adapt}}$ with a substantially smaller coarse space is still
comparable to SHEM$_3$ with three enrichment functions on each interface.

In order to test SHEM on more realistic problems, we consider now the
case where $\Omega$ is subdivided into irregular fine triangles for
the fine mesh and irregular subdomains arising from the graph
partitioning software METIS \cite{Karypis:1998:METIS} for the
subdomains and the coarse mesh. For these experiments, we consider
only the adaptive variant SHEM$_{\tiny\mbox{adapt}}$ and compare it with
the standard multiscale coarse space without enrichment we denote by
MS. For our first example, we modify the distribution given in
Figure~\ref{fig:complicatedex2} by discretizing it with an irregular
mesh such that the largest diameter of the elements is $1/256$. We
then partition the mesh into irregular subdomains using METIS, see
Figure~\ref{fig:complicatedirrMetis}.
 \begin{figure}
   \centering
   \includegraphics[width=\linewidth]{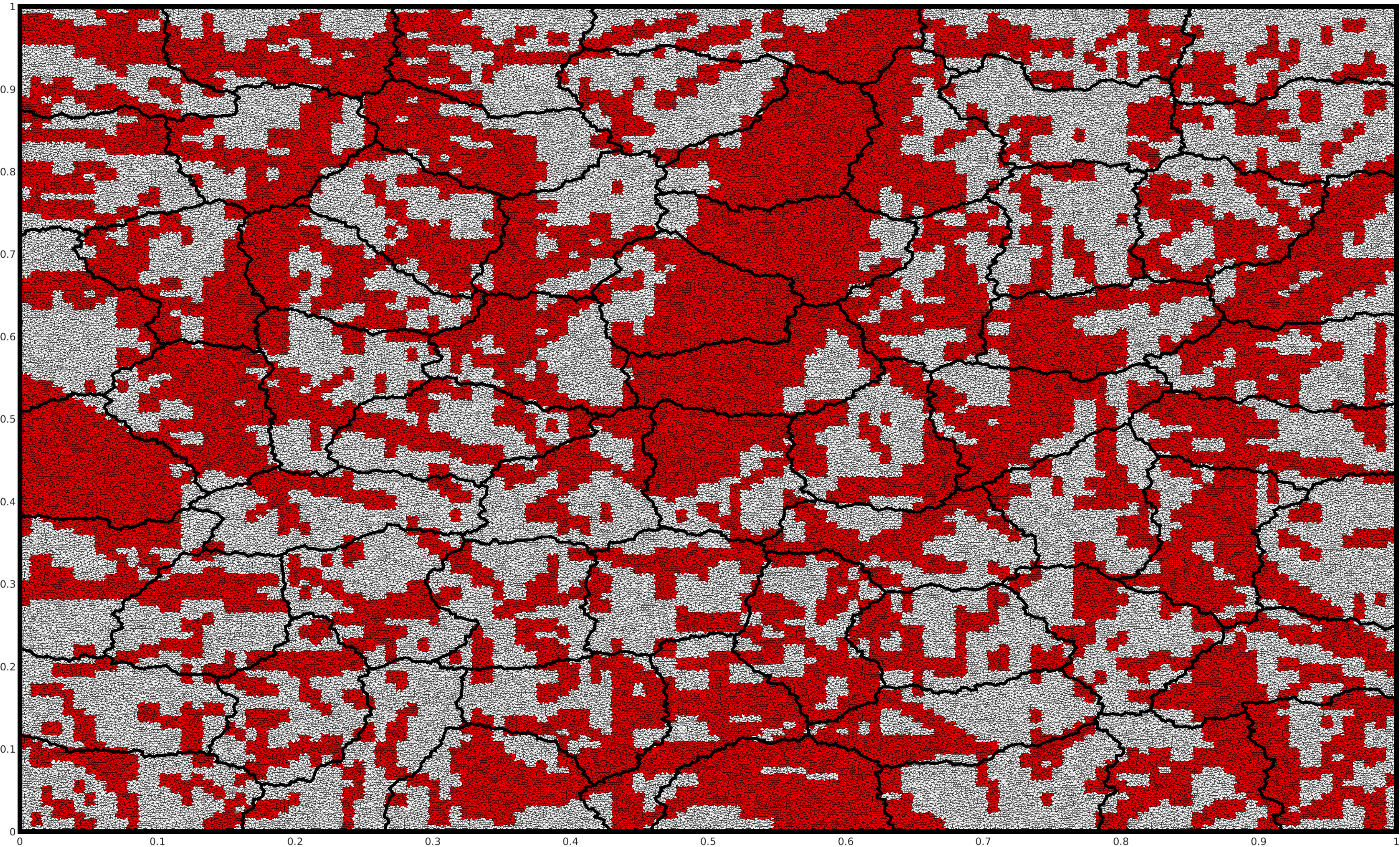}
  \caption{Distribution of $\alpha$ for a geometry discretized with $226918$ irregular triangles and partitioned into 64 irregular subdomains with METIS. The regions marked with red are where $\alpha$ has a large value. The largest subdomain has 1889 dofs.}
  \label{fig:complicatedirrMetis}
\end{figure}
The tolerance for including spectral functions into the coarse space
is set to $\tau=1/32$ and the iteration count and condition number
estimate for \texttt{pcg} are given in
Table~\ref{tbl:complicatedirrMetis} 
\begin{table}\footnotesize
\centering
\begin{tabular}{c c c|| c c}
\hline
Type:&\multicolumn{2}{c||}{MS}&\multicolumn{2}{c}{SHEM$_{\tiny\mbox{adapt}}$}\\
\hline
$\hat\alpha$&\#it. $(\kappa)$&dim.&\#it. $(\kappa)$&dim.\\
\hline
  $10^0$&47 (3.55e1)&101&21 (6.95e0)&676\\
  $10^2$&143 (7.49e2)&101&28 (1.11e1)&738\\
  $10^4$&699 (3.50e4)&101&28 (1.10e1)&738\\
  $10^6$&898 (3.41e6)$^*$&101&30 (1.10e1)&738\\
\hline 
\multicolumn{2}{l}{* Stagnation. }
\end{tabular}
\caption{Comparison of the iteration count and condition number
  estimate for MS (without enrichment) and SHEM$_{\tiny\mbox{adapt}}$ with tolerance $\tau =
  1/32$ for the problem given in
  Figure~\ref{fig:complicatedirrMetis}. The number of coarse dofs for
  both methods is given in the columns next to the iteration count and
  condition number estimate.}
  \label{tbl:complicatedirrMetis}
\end{table} 
where we also report the size of the coarse space. For the second
example, we consider the same fine grid and coarse grid as in the
previous example, but now the coefficients $\alpha$ are taken as the
permeability field from the bottom layer of the SPE 10th comparative
solution project \cite{Christie:2001:Tenth}, see Figure
\ref{fig:SPE10BottomLayer}.
\begin{figure}[h]
   \centering
   \includegraphics[width=\linewidth]{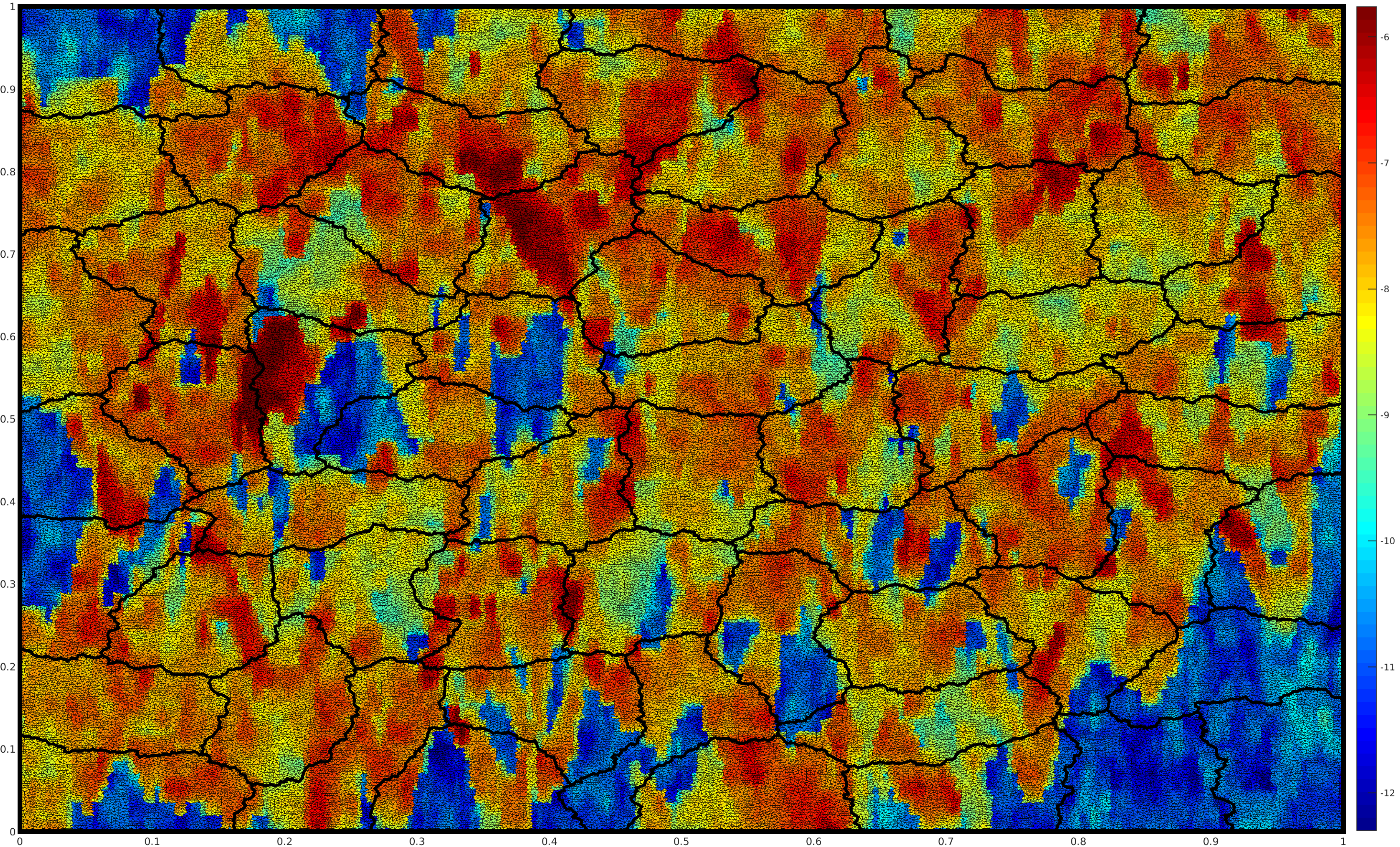}
  \caption{Base 10 logarithm of the permeability field corresponding to the SPE10 Bottom Layer test case for a geometry discretized with $226918$ irregular triangles and partitioned into 64 irregular subdomains with METIS. The largest subdomain has 1889 dofs.}
  \label{fig:SPE10BottomLayer}
\end{figure}
We compare the results of MS with SHEM$_{\tiny\mbox{adapt}}$ for
three different choices of the tolerance $\tau$ and report the
iteration count and condition number estimate for \texttt{pcg} in
Table \ref{tbl:complicatedirrMetis} in addition to the number of
coarse basis functions for each choice of $\tau$.
\begin{table}\footnotesize
\centering
\begin{tabular}{c c c|| c c}
\hline
Type:&\multicolumn{2}{c||}{MS}&\multicolumn{2}{c}{SHEM$_{\tiny\mbox{adapt}}$}\\
\hline
$\tau$&\#it. $(\kappa)$&dim.&\#it. $(\kappa)$&dim.\\
\hline
$\frac{1}{64}$&124 (3.75e2)&101&30 (1.23e1)&489\\
$\frac{1}{32}$&124 (3.75e2)&101&27 (1.09e1)&646\\
$\frac{1}{16}$&124 (3.75e2)&101&23 (7.72e0)&910\\
\hline 
\end{tabular}
\caption{Comparison of the iteration count and condition number
  estimate for MS (without enrichment) and SHEM$_{\tiny\mbox{adapt}}$ for the
  problem given in Figure~\ref{fig:SPE10BottomLayer}. The number of
  coarse dofs for both methods is given in the columns next to the
  iteration count and condition number estimate.}
  \label{tbl:SPE10BottomLayer}
\end{table} 
These examples show that SHEM$_{\tiny\mbox{adapt}}$ performs very well even
for very hard problems and much better than MS that only
uses the initial multiscale coarse space. For all choices of the tolerance
$\tau$, the SHEM$_{\tiny\mbox{adapt}}$ is insensitive to the irregularities of
both the subdomain/coarse partitioning and the fine grid and also with
respect to the high contrast in the underlying material coefficient,
without the dimension of the coarse space becoming larger than the
dimension of the largest subdomain and thus the coarse solve
never becomes the bottleneck of the method.

\section{Conclusion}

We introduced a new harmonically enriched multiscale coarse space
(HEM) based on the multiscale coarse space from
\cite{Graham:2007:DDmult}. Our first enrichment process uses lower
dimensional eigenvalue problems on interfaces between subdomains and
extends these solutions harmonically inside each of the corresponding
subdomains sharing the interface, which leads to the spectral
harmonically enriched multiscale coarse space (SHEM). We showed both
theoretically and numerically that SHEM guarantees robustness with
respect to the discontinuities/variations of the coefficient in the
problem. When we complete the enrichment to the full discrete harmonic
space, the coarse space becomes optimal (OHEM) and the method turns
into a direct solver.

We then proposed three enrichment variants not based on eigenvalue
problems (NSHEM), which simply solve lower dimensional variants of the
original problem on the interfaces between subdomains. Our numerical
experiments showed that NSHEM performs equally well in practice as
SHEM. NSHEM does however not have the eigenvalue information
that can be used as a natural measure to make NSHEM into an adaptive
method that guarantees robustness with respect to the contrast.

The new coarse spaces based on harmonic enrichment are not only 
excellent coarse spaces for the high contrast elliptic model problem we
considered, they are also an ideal choice for many other partial
differential equations, since they are based on the approximation of an
optimal coarse space which contains all coarse space components
necessary to turn the domain decomposition method into a direct
solver. 

\section*{Acknowledgment}
We thank Professor Leszek Marcinkowski for his comments, particularly
for pointing out to us that Definition 2.1 in our earlier version did
not cover the case where elements with large coefficients would touch
subdomain boundaries at only one point.
 \bibliography{multiscalespectralDD}
 \bibliographystyle{plain}
\end{document}